\documentclass[12pt]{amsart}
\usepackage[colorlinks=true, pdfstartview=FitV, linkcolor=blue, citecolor=blue, urlcolor=blue]{hyperref}
\usepackage{amssymb,amscd}
\usepackage{hyphenat}
\calclayout 

\def\quot#1#2{#1/\!\!/#2}

\def\C{\mathbb {C}}

\def\A{\mathcal {A}}

\def\N{\mathbb N}

\def\F{\mathcal {F}}
\def\G{\mathcal {G}}

\def\Der{\operatorname{Der}}

\def\inv{^{-1}}

\def\phi{{\varphi}}

\def\AA{\mathfrak{A}}
\def\CC{C^0}

\def\FF{\mathfrak{F}}

\def\pt{\partial}

\def\E{\mathcal E}
\def\O{\mathcal O}

\def\ci{C^\infty}

\def\Mor{\operatorname{Mor}}

\def\Aut{\operatorname{Aut}}

\def\Id{\operatorname{Id}}
\def\LF{\mathcal {LF}}

\def\LAc{\mathcal{LA}_c}
\def\LA{\mathcal{LA}}
\def\alg{{\operatorname{alg}}}

\def\gf{{\operatorname{fin}}}

\renewcommand{\sl}{{\operatorname{/\!\!/}}}

\newcommand{\red}{{\operatorname{red}}}

\numberwithin{equation}{subsection}

\newtheorem{theorem}[subsection]{Theorem}
\newtheorem{lemma}[subsection]{Lemma}
\newtheorem{proposition}[subsection]{Proposition}
\newtheorem{corollary}[subsection]{Corollary}

\theoremstyle{definition}

\newtheorem{definition}[subsection]{Definition}

\theoremstyle{remark}

\newtheorem{remark}[subsection]{Remark}
\newtheorem{remarks}[subsection]{Remarks}

\title{An Oka principle  for Stein $G$-manifolds}

\author{Gerald W.~Schwarz}
\address{Gerald W.~Schwarz, Department of Mathematics, Brandeis University, Waltham MA 02454-9110, USA}
\email{schwarz@brandeis.edu}

 
\begin{document}
\begin{abstract}  
Let $G$ be a reductive complex Lie group acting holomorphically on  Stein manifolds $X$ and $Y$. Let $p_X\colon X\to Q_X$ and $p_Y\colon Y\to Q_Y$ be the quotient mappings. Assume that we have a biholomorphism $Q:= Q_X\to Q_Y$ and an open cover $\{U_i\}$ of $Q$ and $G$-biholomorphisms $\Phi_i\colon p_X\inv(U_i)\to p_Y\inv(U_i)$ inducing the identity on $U_i$.  
There is a sheaf of groups $\A$ on $Q$ such that the isomorphism classes of all possible $Y$ is the cohomology set $H^1(Q,\A)$. The main question we address is to what extent $H^1(Q,\A)$ contains only topological information. For example, if $G$ acts freely on $X$ and $Y$, then $X$ and $Y$ are principal $G$-bundles over $Q$, and Grauert's Oka principle says that the set of  isomorphism classes of holomorphic principal $G$-bundles   over $Q$ is canonically the same as the set of isomorphism classes of topological principal $G$-bundles over $Q$. We investigate to what extent we have an Oka principle for $H^1(Q,\A)$.
\end{abstract}

\maketitle
\tableofcontents

 \section{Introduction}
  Let $X$ be a Stein $G$-manifold where   $G$ is a complex reductive group. There is a quotient space $Q_X=\quot X G$ (or just $Q$ if $X$ is understood) and surjective morphism $p_X$ (or just $p$) from $X$ to $Q$. Then  $Q$  is a reduced normal Stein space and the fibers of $p$ are canonically affine $G$-varieties (generally, neither reduced nor irreducible) containing precisely one closed $G$-orbit.  For $S$ a subset of $Q$ we denote $p\inv(S)$ by $X_S$ and we abbreviate $X_{\{q\}}$ as $X_q$, $q\in Q$.  We have a sheaf of groups $\A^X$ (or just $\A$) on $Q$   where $\A(U)=\Aut_U(X_U)^G$ is the group of  holomorphic $G$-automorphisms of   $X_U$ which induce the identity map $\Id_U$ on $\quot{X_U}G=U$.
  
  Let $Y$ be another Stein $G$-manifold. In \cite{KLS,KLSOka} we determined sufficient conditions for $X$ and $Y$ to be equivariantly $G$-biholomorphic. Clearly we need that $Q_Y$ is biholomorphic to  $Q_X$, so let us assume that we have fixed an isomorphism   of $Q_Y$ with $Q=Q_X$.  Let us also suppose that there are no local obstructions to a $G$-biholomorphism of $X$ and $Y$ covering $\Id_Q$. (See \cite[Theorem 1.3]{KLSOka} for sufficient conditions for vanishing of the local obstructions.)\ Then there is an open cover $U_i$ of $Q$ and $G$-biholomorphisms $\Phi_i\colon X_{U_i}\to Y_{U_i}$ inducing $\Id_{U_i}$. We say that \emph{$X$ and $Y$ are locally $G$-biholomorphic over $Q$}. Set $\Phi_{ij}=\Phi_i\inv\Phi_j$. Then the $\Phi_{ij}\in\A(U_i\cap U_j)$ are a $1$-cocycle, i.e., an element of $Z^1(Q,\A)$ (we repress explicit mention of the open cover). Conversely, given $\Psi_{ij}\in Z^1(Q,\A)$ (for the same open cover) we can construct a corresponding  complex $G$-manifold $Y$ from the disjoint union of the $X_{U_i}$ by identifying $X_{U_j}$ and $X_{U_i}$ over $U_i\cap U_j$ via $\Psi_{ij}$. By \cite[Theorem 5.11]{KLSOka} the manifold $Y$ is   Stein, and it is obviously locally $G$-biholomorphic to $X$ over $Q$. Let $\Psi_{ij}'$ be another cocycle for $\{U_i\}$ corresponding to the Stein $G$-manifold $Y'$. If $Y'$ is $G$-biholomorphic to $Y$ (inducing $\Id_Q$), then $\Psi_{ij}$ and $\Psi_{ij}'$ give the same class in $H^1(Q,\A)$. Thus $H^1(Q,\A)$ is the set of $G$-isomorphism classes of Stein $G$-manifolds $Y$ which are locally $G$-biholomorphic to $X$ over $Q$ where the $G$-isomorphisms are required to induce the identity on $Q$. 
  
  A fundamental question is whether or not $H^1(Q,\A)$ contains more than topological information. For example, suppose that $G$ acts freely on $X$ so that $X\to Q$ is a principal $G$-bundle. Then $X$ corresponds to an element of $H^1(Q,\E)$ where $\E$ is the sheaf of germs of holomorphic mappings of $Q$ to $G$. By Grauert's famous Oka principle \cite{GrauertFaserungen}, $H^1(Q,\E)\simeq H^1(Q,\E^c)$ where $\E^c$ is the sheaf of germs of continuous mappings of $Q$ to $G$. In other words, the set of isomorphism classes of holomorphic principal $G$-bundles over $Q$ is the same as the set of isomorphism classes of topological principal $G$-bundles over $Q$. The main point of this note is to establish a similar Oka principle in our setting.
  
We define another sheaf  of groups $\A_c$ on $Q$. For $U$ open in $Q$, $\A_c(U)$ consists of    ``strongly continuous'' families  $\sigma=\{\sigma_q\}$ of $G$-automorphisms of the affine $G$-varieties $X_q$, $q\in U$. We define the notion of strongly continuous family in \S \ref{sec:stronglycontinuous}.   The sheaf  $\A$ is a subsheaf of $\A_c$.

Fix an open cover  $\{U_i\}$ of $Q$. Our main theorems are the following (the first of which is a consequence of    \cite[Theorem 1.4]{KLSOka}).

\begin{theorem}\label{thm:A}
Let $\Phi_{ij}$, $\Psi_{ij}\in Z^1(Q,\A)$ and suppose that there are $c_i\in\A_c(U_i)$ satisfying $\Phi_{ij}=c_i \Psi_{ij}c_j\inv$. Then there are $c_i'\in\A(U_i)$ satisfying the same equation.
\end{theorem}

\begin{theorem}\label{thm:B}
Let $\Phi_{ij}\in Z^1(Q,\A_c)$. Then there are $c_i\in\A_c(U_i)$ such that $c_i\Phi_{ij}c_j\inv\in Z^1(Q,\A)$.
\end{theorem}

As a consequence we have the following Oka principle:

\begin{corollary}\label{cor:main}
The canonical map $H^1(Q,\A)\to H^1(Q,\A_c)$ is a bijection.
\end{corollary}

\begin{remark}
 Suppose that $X$ is a smooth affine $G$-variety and that $Z\to Q$ is a morphism of  affine varieties. Then $G$ acts on the fiber product $Z\times_Q X$ and we have the group  $\Aut_{Z,\alg}(Z\times_Q X)^G$ of algebraic $G$-automorphisms of $Z\times_Q X$ which induce the identity on the quotient $Z$.  A scheme $\G$ with projection $\pi\colon \G\to Q$ such that the fibers of $\G$ are groups whose structure depends algebraically on $q\in Q$ is called a \emph{group scheme over $Q$}. (See \cite[Ch,\ III]{KraftSchReductive} for a more precise definition.)\ We say that \emph{the automorphism group scheme of $X$ exists\/} if there is a group scheme $\G$ over $Q$ together with a canonical isomorphism of $\Gamma(Z,Z\times_Q \G)$ and $\Aut_{Z,\alg}(Z\times_Q X)^G$ for all $Z\to Q$. The automorphism group scheme of $X$ exists (and is an affine variety) if, for example, $p\colon X\to Q$ is flat \cite[Ch.\ III Proposition 2.2]{KraftSchReductive}. Assuming $\G$ exists, now consider $X$ as a Stein $G$-manifold and $\G$ as an analytic variety. Then for $U$ open in $Q$, $\A(U)\simeq\Gamma(U,\G)$ and one can show that $\A_c(U)$ is the set of continuous sections of $\G$ over $U$. Thus, in this case,  our theorems reduce to the precise analogues of Grauert's for the cohomology of $\G$ using holomorphic or continuous sections.  
   \end{remark}

For $U$ an open subset of $Q$ we have a topology on $\A_c(U)$ and $\A(U)$ and we  define the notion of a continuous path (or homotopy) in $\A_c(U)$ or $\A(U)$. We establish a result which is well-known in the case of principal bundles but rather non-trivial in our situation. 
\begin{theorem}\label{thm:C}
Let $\Phi_{ij}(t)$ be a homotopy of elements in $Z^1(Q,\A_c)$, $t\in[0,1]$. Then there are homotopies $c_i(t)\in\A_c(U_i)$, $t\in[0,1]$, such that $\Phi_{ij}(t)=c_i(t)\Phi_{ij}(0)c_j(t)\inv$.
Hence $\Phi_{ij}(t)\in H^1(Q,\A_c)$ is independent of $t$.
\end{theorem}

 \begin{theorem}\label{thm:D}
Let $\Phi_{ij}(t)\in Z^1(Q,\A_c)$ be a homotopy, $t\in[0,1]$, where the $\Phi_{ij}(0)$ and $\Phi_{ij}(1)$ are holomorphic. Then there is a homotopy   $\Psi_{ij}(t)\in Z^1(Q,\A)$ with $\Psi_{ij}(0)=\Phi_{ij}(0)$ and $\Psi_{ij}(1)=\Phi_{ij}(1)$.
\end{theorem}

Here is an outline of this paper. In \S \ref{sec:background} we recall Luna's slice theorem and related results. In \S \ref{sec:stronglycontinuous} we define the sheaf of groups $\A_c$  as well as  a corresponding sheaf of Lie algebras $\LAc$.     In \S \ref{sec:logarithms} we show that sections of $\A_c$ sufficiently close to the identity are the exponentials of   sections of $\LA_c$.  In \S \ref{sec:homotopiesAc} we establish our main technical result (Theorem \ref{thm:mainAc}) about homotopies in $\A_c$. We prove Theorem \ref{thm:A} and Theorem \ref{thm:D} as well as a preliminary version of Theorem \ref{thm:C}. In \S \ref{sec:proveB} we establish Theorem \ref{thm:B} and use it to prove Theorem \ref{thm:C}. Finally, let $X$ and $Y$ be locally $G$-biholomorphic over $Q$. We establish a theorem giving necessary and sufficient conditions for a $G$-biholomorphism from $X_U\to Y_U$ over $\Id_U$, where $U\subset Q$ is Runge, to be  the limit of the restrictions  to $X_U$ of  $G$-biholomorphisms from $X$ to $Y$ over $\Id_Q$. 
   \begin{remark}
  In \cite{KLS, KLSOka} we also  consider  $G$-diffeomorphisms $\Phi$ of $X$   which induce the identity over $Q$ and are \emph{strict\/}. This means that   the restriction of $\Phi$ to $X_q$, $q\in Q$, induces an algebraic  $G$-automorphism of    $(X_q)_\red$  where ``$\red$'' denotes reduced structure. One can adapt the techniques developed here to prove the analogues of our main theorems for strong $G$-homeomorphisms replaced by strict $G$-diffeomorphisms.
\end{remark}

  \smallskip\noindent
\textit{Acknowledgement.}  I thank F.~Kutzschebauch and F.~L\'arusson for our collaboration on \cite{KLS} and \cite{KLSOka} which led to this paper. 

\section{Background}  \label{sec:background}

For details of what follows see \cite{Luna} and \cite[Section~6]{Snow}.  Let $X$ be a   Stein manifold with a holomorphic action of a reductive complex Lie group $G$.  The categorical quotient $Q_X=X\sl G$ of $X$ by the action of $G$ is the set of closed orbits in $X$ with a reduced Stein structure that makes the quotient map $p_X\colon X\to Q_X$ the universal $G$-invariant holomorphic map from $X$ to a Stein space. The quotient $Q_X$ is normal. When $X$ is understood, we drop the subscript $X$ in $p_X$ and $Q_X$.      If $U$ is an open subset of $Q$, then 
$p^*$ induces isomorphisms of $\C$-algebras $\O_X(X_U)^G \simeq \O_Q(U)$  and $\CC(X_U)^G\simeq\CC(U)$. We say that a subset of $X$ is \textit{$G$-saturated\/} if it is a union of fibers of $p$. If $X$ is an affine  $G$-variety, then $Q$ is just the complex space corresponding to the affine algebraic variety with coordinate ring  $\O_\mathrm{alg}(X)^G$.

 Let $H$ be a reductive subgroup of $G$ and let $B$ be an $H$-saturated neighborhood of the origin of an $H$-module $W$. We always assume   that $B$ is Stein, in which case $B\sl H$ is also Stein. Let $G\times^HB$ (or $T_B$) denote the quotient of $G\times B$ by the (free) $H$-action sending $(g,w)$ to $(gh\inv,hw)$ for $h\in H$, $g\in G$ and $w\in B$. We denote the image of $(g,w)$ in $G\times^HB$ by $[g,w]$.  
 
  Let $Gx$ be a closed orbit in $X$. Then the   isotropy group  $G_x$ is reductive and the \emph{slice representation at $x$\/}   is the action of $H=G_x$ on $W=T_xX/T_x(Gx)$.  By the slice theorem, there is a $G$-saturated neighborhood of $Gx$ which is $G$-biholomorphic to $T_B$ where $B$ is an $H$-saturated neighborhood of $0\in W$.

  \section{Strongly continuous homeomorphisms and vector fields}\label{sec:stronglycontinuous}
  
  The group $G$ acts on $\O(X)$, $f\mapsto g\cdot f$, where $(g\cdot f)(x)=f(g\inv x)$, $x\in X$, $g\in G$, $f\in\O(X)$.
Let $\O_\gf(X)$ denote the set of holomorphic functions  $f$ such that the span of $\{g\cdot f\mid g\in G\}$ is finite dimensional. They are called the  \emph{$G$-finite  holomorphic functions on $X$} and obviously form an $\O(Q)=\O(X)^G$-algebra.  If $X$ is a smooth affine $G$-variety, then the techniques of \cite[Proposition 6.8, Corollary 6.9]{Schwarz1980} show that for $U\subset Q$ open and Stein we have
$$
\O_\gf(X_U)\simeq \O(U)\otimes_{\O_\alg(Q)}\O_\alg(X).
$$
Let $V$ be the direct sum of pairwise non-isomorphic non-trivial $G$-modules $V_1,\dots,V_r$. Let $\O(X)_V$ denote the elements of $\O_\gf(X)$ contained in a copy of $V$.  If $H$ is a reductive subgroup of $G$ and $W$ an $H$-module, we similarly define $\O_\alg(T_W)_V$. 
Then for $B$ an $H$-saturated neighborhood of $0\in W$, $\O_\alg(T_W)_V$ generates $\O(T_B)_V$ over $\O(B)^H$. By Nakayama's Lemma, $f_1,\dots,f_m\in\O(X)_V$ restrict to minimal   generators of the $\O(U)$-module $\O(X_U)_V$   for some neighborhood $U$ of $q\in Q$  if and only if the restrictions of the $f_i$ to $X_q$ form a basis of  $\O(X_q)_V=\O_\alg(X_q)_V$. Thus by the slice theorem,  the sheaf of algebras of $G$-finite holomorphic functions is locally finitely generated as an algebra over $\O_Q$. 

  \begin{definition}\label{def:standard-generators} Let $U\subset Q$ be relatively compact. Then there is a $V$  as above such that  the  $\O(X_U)_{V_j}$  are finitely generated over $\O(U)$ and generate $\O_\gf(X_U)$ as $\O(U)$-algebra.   Let $f_1,\dots,f_n$ be a   generating set of $\oplus\O(X_U)_{V_j}$ with each $f_i$ in some $\O(X_U)_{V_j}$. Then we call   $\{f_i\}$ a \emph{standard generating set    of $\O_\gf(X_U)$}. When $U=T_B\sl G$  as before, we always assume that our standard generators are the restrictions of homogeneous elements of  $\O_\alg(T_W)$.
  \end{definition}

 Let $U\subset Q$, $V$ and $\{f_1,\dots,f_n\}$ be as above.  We say that a $G$-equivariant homeomorphism  
 $\Psi:X_U\to X_U$ is \emph{strong} if it lies over the identity of $U$  and   $\Psi^*f_i=\sum_{j} a_{ij}f_j$   where  the $a_{ij}$  are in $\CC(X_U)^G\simeq\CC(U)$. We also require that the $a_{ij}(q)$ induce a  $G$-isomorphism of $\O(X_q)_V$ for all $q\in U$.  Then  $\Psi$ induces an algebraic isomorphism $\Psi_q\colon X_q\to X_q$  for all $q\in U$.  It is easy to see that the definition does not depend on our choice of $V$ and the generators $f_i$. We call $(a_{ij})$ a \emph{matrix associated to $\Psi$\/}. Using a partition of unity on $U$ it is clear that $\Psi$ is strong if and only if it is strong in a neighborhood of every $q\in Q$. In a neighborhood of any particular $q$, we may assume that the $f_i$ restrict to a basis of $\O(X_q)_V$, in which case $(a_{ij})$ is invertible in a neighborhood of $q$. Then $\Phi\inv$ has matrix $(a_{ij})\inv$ near $q$. Thus if $\Phi$ is strong, so is $\Phi\inv$. Let $\A_c(U)$ denote the group of strong $G$-homeomorphisms of $X_U$  for $U$ open in $Q$. Then $\A_c$ is a sheaf of groups on $Q$.

 We say that a vector field $D$ on $X_U$ is \emph{formally holomorphic\/} if it annihilates the antiholomorphic functions on $X_U$. Let $D$ be a continuous formally holomorphic vector field on $X_U$, $G$-invariant, annihilating $\O(X_U)^G$.  We say that \emph{$D$ is strongly continuous\/} (and write $D\in\LAc(U)$) if for any $q\in U$ there is a neighborhood $U'$ of $q$ in $U$ and a standard generating set $f_1,\dots,f_n$ for $\O_\gf(X_{U'})$ such that $D(f_i)=\sum d_{ij}f_j$ where the $d_{ij}$ are in $\CC(U')$. We say that \emph{$D$ has matrix $(d_{ij})$ over $U'$\/}. The matrix   is usually not unique. Clearly our definition of $\LAc(U)$ is independent of the choices made. We denote the corresponding sheaf by $\LAc$. 
 
 \begin{remark}\label{rem:Dcomplete} Let $D\in\LAc(U)$ and $q\in Q$. Then $D$ is tangent to $F=X_q$ and acts algebraically on $\O_\alg(F)$, hence lies in the space of $G$-invariant derivations $\Der_\alg(F)^G$ of $\O_\alg(F)$. Since $\Der_\alg(F)^G$ is the Lie algebra of the algebraic group $\Aut(F)^G$, the restriction of $D$ to $F$  can be integrated for all time. It follows that $D$ is a complete vector field.  
 \end{remark}

Let $U$ be   open in $Q$, let $\epsilon>0$ and let $K$ be a compact subset of $U$. Let ${\mathbf f}=\{f_1,\dots,f_n\}$ be a standard generating set of $\O_\gf(X_{U'})$ where $U'$ is a neighborhood of $K$.  Define
$$
\Omega_{K,\epsilon,\mathbf f}=\{\Phi\in\A_c(U): ||(a_{ij})-I||_K<\epsilon\}
$$ 
where $(a_{ij})$ is some matrix associated to $\Phi$. Here $||(a_{ij})-I||_K$ denotes the supremum of the matrix norm of $(a_{ij})-I$ over $K$. Let $\mathbf f'=\{f_1',\dots,f_m'\}$ be another standard generating set defined on a neighborhood of $K$ in $U$.

\begin{lemma}\label{lem:topology-well-defined}
Let $\epsilon'>0$. Then there is an $\epsilon>0$ such that $\Omega_{K,\epsilon,\mathbf f}\subset\Omega_{K,\epsilon',\mathbf f'}$.
\end{lemma}
\begin{proof} We may assume that the $f_i$ and $f_j'$ are standard generating sets of $\O_\gf(U)$.
There are polynomials $h_i$   with coefficients in $\O(U)$ such that   $f_i'=h_i(f_1,\dots,f_n)$, $1\leq i\leq m$. We may assume that $\{f_1',\dots,f_s'\}$ are the $f'_i$ corresponding to an irreducible $G$-module $V_t$. Let $\Phi\in\Omega_{K,\epsilon,\mathbf f}$ with corresponding matrix $(a_{uv})$ such that $||(b_{uv})||_K<\epsilon$ where $(b_{uv})=(a_{uv})-I$. Let $r_i$ be the degree of $h_i$. Then for $1\leq i\leq s$ we have
$$
(\Phi^*f_i')-f_i'=h_i(\Phi^*f_1,\dots,\Phi^*f_n)-h_i(f_1,\dots,f_n)=\sum_{k,l=1}^n   b_{kl}p_{kl}M_{kl}(f_1,\dots,f_n)
$$
where the $p_{kl}$ are polynomials in the $a_{uv}$ of degree at most $r_i-1$ and the $M_{k,l}$ are polynomials in the $f_j$ with coefficients in $\O(U)$ which are independent of the $a_{uv}$ and $b_{uv}$. Since $\Phi^*f_i'$ is a covariant corresponding to $V_t$, we can project the $M_{kl}$ to $\O(X_U)_{V_t}$ in which case we get   $\sum_{j=1}^s N_{jkl}f_j'$ where the $N_{jkl}$ are in $\O(U)$ and independent of the $a_{uv}$ and $b_{uv}$. Hence
$$
(\Phi^*f_i')-f_i'=\sum_{j=1}^s \sum_{k,l=1}^n b_{kl}N_{jkl}p_{kl}f_j'.
$$
Since the $N_{jkl}p_{kl}$ are bounded on $K$, choosing $\epsilon$ sufficiently small, we can force the terms $\sum_{k,l=1}^n b_{kl}N_{jkl}p_{kl}$ to be close to 0. Hence there is an $\epsilon>0$ such that $\Omega_{K,\epsilon,\mathbf f}\subset\Omega_{K,\epsilon',\mathbf f'}$.
\end{proof}

By the lemma, we get the same neighborhoods of the identity in $\A_c(U)$ from any standard generating set of $\O_\gf(X_{U'})$ where $U'$ is a neighborhood of $K$. Thus we can talk about neighborhoods of the identity without specifying the $\mathbf f$ in question. We then have a well-defined topology on $\A_c(U)$ where $\Phi$ is close to $\Phi'$ if $\Phi'\Phi\inv$ is close to the identity.

Let $U$, $K$ and the $f_i$ be as above. Define
 $$
 \Omega'_{K,\epsilon,\mathbf f}=\{D\in\LAc(U)\mid D(f_i)=\sum d_{ij}f_j\text{ and }  ||(d_{ij})||_K<\epsilon\}
 $$
 where $D$ has  (continuous) matrix $(d_{ij})$   defined on a neighborhood of $K$.
  As before, the $\Omega'_{K,\epsilon,\mathbf f}$ give a basis of neighborhoods of $0$  and define a topology on $\LAc(U)$, independent of the choice of  $\mathbf f$.

\begin{proposition}\label{prop:Dbasics} Let $U$ be open in $Q$ and let  $\{f_1,\dots,f_n\}$ be a standard generating set for $\O_\gf(U)$.
\begin{enumerate}
\item Let $D$ be a $G$-invariant formally holomorphic vector field on $X_U$ which annihilates $\O(U)$ such that $D(f_i)=\sum d_{ij}f_j$ for $d_{ij}\in\CC(U)$. Then $D$ is continuous, i.e., $D\in\LAc(U)$.
\item  $\LAc$ is a sheaf of Lie algebras and a module over the sheaf of germs of continuous functions on $Q$.
\item $\exp\colon \LAc(U)\to\A_c(U)$ is continuous.
\item $\LAc(U)$ is a Fr\'echet space.
\end{enumerate}
\end{proposition}

\begin{proof}
Let $D$ be as in (1) and let $x_0\in X_U$.   There is a subset, say $f_1,\dots,f_r$, of the $f_i$ and holomorphic invariant functions $h_{r+1},\dots,h_s$ such that the $z_i=f_i-f_i(x_0)$ and $z_j=h_j-h_j(x_0)$ are local holomorphic coordinates at $x_0$. Then, near $x_0$, $D$ has the form $\sum a_i\pt/\pt z_i$ where each $a_i=D(f_i)$ is continuous. Hence $D$ is continuous giving (1). Let $D$, $D'\in\LAc(U)$ with matrices $(d_{ij})$ and $(d'_{ij})$. Let $(e_{ij})$ be their matrix bracket. Then $[D,D']$ is $G$-invariant, annihilates $\O(U)$ and sends $f_i$ to $\sum e_{ij}f_j$. Hence  we have (2).  Part (3) is clear.

The topology on $\LAc(U)$, $U$ open in $Q$, is defined by countably many seminorms, hence $\LAc(U)$ is a metric space and it is Fr\'echet if it is complete.   Let $D_k$ be a Cauchy sequence in $\LAc(U)$.  Let $K\subset U$ be a compact neighborhood of $q\in U$. There are matrices $(d^k_{ij})$ of elements of $\CC(K)$ such that $D_k(f_i)=\sum d_{ij}^kf_j$ over $K$. Since $\{D_k\}$ is Cauchy, we may assume that $||(d^k_{ij})-(d^l_{ij})||_{K}<1/m$ for $k$, $l>N_m$, $m\in \N$.   Then $\lim_{k\to\infty}d^k_{ij}=d_{ij}\in\CC(K)$ for all $i$, $j$. It follows that the pointwise limit of the $D_k$ exists and is a formally holomorphic vector field $D$ annihilating the invariants such that  $D(f_i)=\sum d_{ij}f_j$.
By (1), $D$ is of type $\LAc$ over the interior of $K$. It follows that $\LF(U)$ is complete.                                                                                                                              
\end{proof}

 \section{Logarithms in $\A_c$}\label{sec:logarithms}
  Let $U$ be an open subset of $Q$ isomorphic to $T_B=G\times^HB$ where $H$ is a reductive subgroup of $G$ and $B$ is an $H$-saturated neighborhood of the origin in an $H$-module $W$. Let   $f_1,\dots,f_n$ be a standard generating set for $\O_\gf(X_U)^G$ consisting of the restrictions to $X_U$ of homogeneous polynomials in $\O_\alg(T_W)$. Consider polynomial relations of the $f_i$ with coefficients in $\O(U)$. These are generated by the  relations with coefficients in $\O_\alg(T_W)^G$. Let  $h_1,\dots,h_m$ be generating relations of this type. Let $N$ be a bound for the degree of the $h_j$. Now take the covariants which correspond to all the irreducible $G$-representations occurring in the span of the monomials of degree at most $N$ in the $f_i$. Let $\{f_\alpha\}$ be a set of generators for these covariants and let $K\subset U$ be compact. Let $\Phi\in\A_c(U)$.
Then $\Phi^*f_\alpha=\sum c_{\alpha,\beta}f_\beta$ where the $c_{\alpha,\beta}\in\CC(U)$. We also have that $\Phi^*f_i=\sum a_{ij}f_j$ where the $a_{ij}\in\CC(U)$. We fix      a neighborhood $\Omega$ of the identity in $\A_c(U)$ such that $\Phi\in\Omega$ implies that  $||(c_{\alpha,\beta})-I||_K<1/3$ and that $||(a_{ij})-I||_K<1/2$. For $\Phi\in\Omega$ let $\Lambda'$ denote $\Id-\Phi^*$. Then the formal power series $S(\Lambda')$ for $\log\Phi^*$ is   $-\Lambda'-(1/2)(\Lambda')^2-1/3(\Lambda')^3-\cdots$. 

Now we restrict to a fiber $F=X_q$,  $q\in K$. Let $M$ denote the span of the covariants $f_\alpha$ restricted to $F$. Then $M$ is finite dimensional and we give it the usual euclidian topology. Let $\Lambda$ denote the restriction of $\Lambda'$ to $M$.

\begin{lemma}\label{lem:converges}
Let $m\in M$. Then the series $S(\Lambda)(m)$ converges in $M$.
\end{lemma}

\begin{proof}
We have $m=\sum a_\alpha f_\alpha|_F$ where the $a_\alpha\in\C$. Then $\Lambda(m)=\sum_{\alpha,\beta}(\delta_{\alpha,\beta}-c_{\alpha,\beta}(q))a_\beta f_\beta|_F$. Let $C$ denote $(c_{\alpha,\beta}(q))$. Then $||I-C||<1/3$. By induction, $\Lambda^k$ acts on $\sum_\alpha a_\alpha f_\alpha|_F$ via the matrix $(I-C)^k$, where $||(I-C)^k||<(1/3)^k$.
Let 
$$
C'=-\sum_{k=1}^\infty (I-C)^k.
$$
Then  
$S(\Lambda)(m)$ converges to $\sum_{\alpha,\beta}C'_{\alpha,\beta}a_\beta f_\beta|_F\in M$.
\end{proof}

For $f\in M$, define $D(f)$ to be the limit of $S(\Lambda)(f)$. Then $D$ is a $G$-equivariant linear endomorphism of $M$. 

\begin{proposition}\label{prop:derivation}
Suppose that $m_1$, $m_2$ and $m_1m_2$ are in $M$. Then 
$$
D(m_1m_2)=D(m_1)m_2+m_1D(m_2).
$$
\end{proposition}

\begin{proof}
By \cite[Proof of Theorem 4]{Praagman} 
$$
\Lambda^k(m_1m_2)=\sum_{l=0}^{2k}\sum_{n=0}^\ell c_{k\ell n} \Lambda^n(m_1)\Lambda^{\ell-n}(m_2)
$$
where  $c_{k\ell n}$ is the coefficient  of $x^ny^{\ell-n}$ in $(x+y-xy)^k$. We know that  $\Lambda^k$ is given by the action of the matrix $(I-C)^k$ where   $||I-C||<1/3$.
The series  $\sum 1/k(x+y-xy)^k$ converges absolutely when $x$ and $y$ have absolute value at most $1/3$. Thus we may make a change in the order of summation:
$$
D(m_1m_2)=\sum_{k=1}^\infty \sum_{l=0}^{2k}\sum_{n=0}^\ell \frac 1kc_{k\ell n}\Lambda^n(m_1)\Lambda^{\ell-n}(m_2)=\sum_{\ell=0}^\infty\sum_{n=0}^\ell\sum_{k=1}^\infty \frac 1k c_{k\ell n}\Lambda^n(m_1)\Lambda^{\ell-n}(m_2).
$$
By \cite[Proof of Theorem 4]{Praagman}   the (actually finite) sum $\sum_{k=1}^\infty (1/k) c_{k\ell n}$ equals $0$ unless we have $\ell>0$ and ($n=0$   or $n=\ell$), in which case the value is $1/\ell$.
Hence 
$$
D(m_1m_2)=-\sum_{\ell=1}^\infty\frac 1\ell (\Lambda^\ell(m_1)m_2+m_1\Lambda^\ell(m_2))=D(m_1)m_2+m_1D(m_2).
$$
\end{proof}

\begin{proposition}\label{prop:Dexists}
Let $D\colon M\to M$ be as above. Then $D$ extends to a $G$-equivariant   derivation of $\O_\alg(F)$.
\end{proposition}

\begin{proof}
Let $R=\C[z_1,\dots,z_n]$. We have a surjective morphism $\rho\colon R\to\O_\alg(F)$ sending $z_i$ to $f_i|_F$, $i=1,\dots,n$. The kernel $J$ of $\rho$ is generated by
polynomials of degree at most $N$ (the $h_j$ considered as elements of $R$). Let $E$ denote the derivation of $R$ which sends $z_i$ to $\sum d_{ij}'z_j$ where $(d_{ij}')$ is the logarithm of $(a_{ij}(q))$. Recall that $||(a_{ij})-I||_K<1/2$.  By construction, $E$ on the span of the $z_i$ is the pull-back of  $D$ on the span of the $f_i|_F$. By Proposition  \ref{prop:derivation}, $E$ restricted to polynomials of degree at most $N$ is the pull-back of $D$ restricted to polynomials of degree at most $N$ in the $f_i|_F$. Hence $E$ preserves the span of the elements of degree at most $N$ in $J$. Since these elements generate  $J$, we see that $E$ preserves $J$. Hence $E$ induces a derivation of $R/J$, i.e., $D$ extends to a $G$-invariant derivation of $\O_\alg(F)$.
\end{proof}

\begin{corollary}\label{cor:logexists} Let $\Phi\in\Omega$ and let $U'$ denote the interior of our compact subset $K\subset U$.
There is a $D\in\LAc(U')$ such that $\exp(D)=\Phi|_{X_{U'}}$. The mapping $\Omega\ni\Phi\to D\in\LAc(U')$ is continuous.
\end{corollary}

\begin{proof} For $q\in U'$, let $D_q$ be the $G$-equivariant derivation of $\O_\alg(X_q)$ constructed above. Let $D$ be the vector field on $X_{U'}$ whose value on $X_q$ is $D_q$, $q\in U'$. Then $D(f_i)=\sum d_{ij}f_j$ where $(d_{ij})=\log (a_{ij})$. By Proposition \ref{prop:Dbasics}, $D\in\LAc(U')$. By construction, $\exp(D_q)=\Phi_q$ for all $q\in U'$. Hence $\exp D=\Phi|_{X_{U'}}$. The continuity of $\Phi\mapsto D$ is clear since $(d_{ij})=\log (a_{ij})$.
\end{proof}
 
  \begin{definition}\label{def:admitslog} Let $U\subset Q$ be open and let $f_1,\dots,f_n$ be a standard generating set of $\O_\gf(X_U)$. Let $U'\subset U$ be open with $\overline U'\subset U$. We say that $\Phi\in\A_c(U)$ \emph{admits a logarithm in\/}  $\LAc(U')$ if the following hold.
 \begin{enumerate}
\item $\Phi^*f_i=\sum a_{ij}f_j$ where $||(a_{ij})-I||_{\overline{U'}}<1/2$.
\item There is a $D\in\LAc(U')$ such that  $D(f_i)=\sum d_{ij}f_j$ on $X_{U'}$ where $(d_{ij})=\log (a_{ij})$. 
\end{enumerate}
 Note that $(a_{ij})$ is not unique. The condition is that some $(a_{ij})$ corresponding to $\Phi$ satisfies (1) and (2).
\end{definition}
 \begin{remarks}\label{rem:admitslog}
  The formal series $\log\Phi^*$, when applied to any $f_i$, converges to $D(f_i)$. Hence $D$ is independent of the choice of $(a_{ij})$. Properties (1) and (2) imply that $\exp D=\Phi$ over $U'$. Note that $||(d_{ij})||_{\overline U'}<\log 2$ and $(d_{ij})$ is the unique matrix satisfying this property whose exponential is $(a_{ij})$. 
 \end{remarks}
 
 Corollary \ref{cor:logexists} and its proof imply the following result.
 
\begin{theorem}\label{thm:log} Let $K\subset U\subset Q$ where $K$ is compact and $U$ is open. Then there is a neighborhood $\Omega$ of the identity in $\A_c(U)$ and a neighborhood  $U'$ of $K$ in $U$ such that every $\Phi\in\Omega$ admits a logarithm $D=\log\Phi$ in $\LAc(U')$. The mapping $\Phi\to\log\Phi$ is continuous.
\end{theorem}

 \begin{corollary}\label{cor:Ac_complete}
Let $\Phi_n$ be a Cauchy sequence in $\A_c(U)$. Then $\Phi_n\to\Phi\in\A_c(U)$.
\end{corollary}

\begin{proof} Since this is a local question, we can assume that we have a standard generating set $\{f_i\}$ for $\O_\gf(U)$. Let $q\in U$  and let $U'$ be a relatively compact neighborhood of $q$ in $U$. Then  there is a neighborhood $\Omega$ of the identity in $\A_c(U)$  such that any $\Psi\in\Omega$   admits a logarithm in $\LAc(U')$. Let $\Omega_0$ be a smaller neighborhood of the identity with $\overline{\Omega}_0\subset\Omega$. There is an $N\in\N$ such that $n\geq N$ implies that $\Phi_N\inv\Phi_n\in\Omega_0$, hence $\log(\Phi_N\inv\Phi_n)=D_n\in\LAc(U')$, and $D_n$ converges to an element  $D\in\LAc(U')$ by Proposition \ref{prop:Dbasics}.  Set   $\Phi=\Phi_N\exp D\in\A(U')$. Since $\exp D_n=\Phi_N\inv\Phi_n$ over $U'$ we have $\Phi_n\to\Phi$ in $A_c(U')$.
\end{proof}

\section{Homotopies in $H^1(Q,\A_c)$}\label{sec:homotopiesAc}

We  establish our main technical result concerning homotopies in $H^1(Q,\A_c)$. We give   proofs of Theorems \ref{thm:A} and \ref{thm:D} and a special case of Theorem \ref{thm:C}.

 Let $\Phi(t)\in\A_c(U)$, $t\in C$, where $C$ is a topological space. We say that $\Phi(t)$ is continuous if relative to a standard generating set, $\Phi(t)$ has corresponding matrices $(a_{ij}(t,q))$ where each $a_{ij}$ is continuous in $t$ and $q\in U$.  (It is probably  false that every continuous map $C\to\A_c(U)$ is continuous in our sense.)\ Let $\AA_c(U)$ denote the set of all  continuous paths $\Phi(t)\in\A_c(U)$, $t\in[0,1]$,  starting at the identity. We have a topology on $\AA_c(U)$ as in \S \ref{sec:stronglycontinuous} and  $\AA_c(U)$ is a topological group. When we talk of homotopies in $\AA_c(U)$ we mean that the corresponding families with parameter space  $[0,1]^2$  are continuous as above. We define continuous families of elements of $\LAc(U)$ similarly. One defines $\AA(U)$ similarly to $\AA_c(U)$ where, of course, the relevant $a_{ij}(t,q)$ are required to be holomorphic in $q$ and continuous in $t$.

Here is our main technical result about $\AA_c$.
\begin{theorem}\label{thm:mainAc}
\begin{enumerate}
\item The topological group $ \AA_c(Q)$ is pathwise connected.
\item If $U\subset Q$ is  open, then $\AA_c(Q)$ is dense in $\AA_c(U)$.
\item $H^1(Q,\AA_c)=0$.
\end{enumerate}
\end{theorem}
\begin{proof}
Let $\Phi(t)$ be an element of $\AA_c(Q)$. Since $\{0\}$ is a deformation retract of $[0,1]$, there is a homotopy $\Phi(s,t)$ with $\Phi(0,t)=\Phi(t)$ and $\Phi(1,t)$ the identity automorphism. Hence we have (1). For (2), let $\Phi\in\AA_c(U)$. Let $K$ be a compact subset of $U$ and $U'$ a relatively compact neighborhood of $K$ in $U$. It follows from Theorem \ref{thm:log} that there are $0=t_0<t_1<\dots<t_m=1$ and continuous families   $D_j(s)$ in $\LAc(U')$  for $s\in[0,t_{j+1}-t_j]$ such that, over $U'$,  $\Phi(s+t_{j})=\Phi({t_j})\exp(D_j(s))$, $s\in[0,t_{j+1}-t_j]$, $j=0,\dots,m-1$.  Multiplying by a cutoff function, we can assume that the $D_j(s)$ are in $\LAc(Q)$. Then our formula gives an element of $\AA_c(Q)$ which restricts to $\Phi$ on a neighborhood of $K$ and we have (2).

Let $K\subset Q$ be compact which is of the form $K'\cup K''$ where $K'$ and $K''$ are compact. Let $U=U'\cup U''$ be a neighborhood of $K$ where $K'\subset U'$, $K''\subset U''$. Let $\Phi(t)$ be in $Z^1(U,\AA_c)$ for the open covering $\{U',U''\}$ of $U$. Then $\Phi(t)$ is just an element in $\AA_c(U'\cap U'')$. By (2) we can write $\Phi=\Psi_1\Psi_2\inv$ where $\Psi_1$ is defined over $Q$ (hence over $U'$) and $\Psi_2$ is close  to the identity over $K'\cap K''$. Then $\Psi_2(t)=\exp D(t)$ where $D(t)\in\LAc(U'\cap U'')$ is a continuous family and $D(0)=0$. Using a cutoff function again, we can find $D_0(t)\in\LAc(Q)$ which equals $D(t)$ in a neighborhood of $K'\cap K''$ and vanishes when $t=0$. We have $\Phi=\Psi_1\Psi_2\inv$ where $\Psi_2\inv$ is the exponential of $D_0(t)$. Thus the cohomology class of $\Phi$ becomes trivial if we replace $U'$ and $U''$ by slightly smaller neighborhoods of $K'$ and $K''$. Let $H^1(K,\AA_c)$ denote the direct limit of $H^1(U,\AA_c)$ for $U$ a neighborhood of $K$. As in   \cite[\S 5]{Cartan58}, our result above shows that there is a sequence of compact sets $K_1\subset V_2\subset K_2\dots$ with $V_n$   the interior of $K_n$, $Q=\bigcup K_n$ and $H^1(K_n,\AA_c)=0$ for all $n$.

Let $\{U_i\}$ be an open cover of $Q$ and $\Phi_{ij}\in\AA_c(U_i\cap U_j)$ a cocycle. There are $c_{i}^n\in\AA_c(U_i\cap V_n)$ such that $\Phi_{ij}=(c_i^n)\inv c_j^n$ on $U_i\cap U_j\cap V_n$. Thus $c_i^{n+1}(c_i^n)\inv=c_j^{n+1}(c_j^n)\inv$ on $U_i\cap U_j\cap V_n$. The $c_i^{n+1}(c_i^n)\inv$ define a section $d\in\AA_c(V_n)$. By (2) there is a section $d'$ of $\AA_c(Q)$ which is arbitrarily close to $d$ on $K_{n-1}$.  Replace each $c_i^{n+1}$ by $(d')\inv c_i^{n+1}$. Then $c_i^{n+1}$ is very close to $c_i^n$ on $K_{n-1}$ and we can arrange that   the limit as $n\to\infty$ of the $c_i^n$ converges on every compact subset to $c_i\in\AA_c(U_i)$ such that $\Phi_{ij}=c_i\inv c_j$. We have used Corollary \ref{cor:Ac_complete}. This completes the proof of (3).
\end{proof}

Note that (3) says that for any homotopy of a cocycle $\Phi_{ij}(t)$ starting at the identity there are $c_i(t)\in\A_c(U_i)$ such that $\Phi_{ij}(t)=c_i(t)\inv c_j(t)$ for all $t\in[0,1]$. Hence $\Phi_{ij}(t)$ is the trivial element in $H^1(Q,\A_c)$ for all $t$. We now use a trick to show a similar result   if we only assume that $\Phi_{ij}(0)\in Z^1(Q,\A)$.

 Let $\Psi_{ij}\in Z^1(Q,\A)$ for some open cover $\{U_i\}$ of $Q$.  By \cite[Theorem 5.11]{KLSOka}, there is a Stein $G$-manifold $Y$ with quotient $Q$ corresponding to the $\Psi_{ij}$. Let $X_i=X_{U_i}$ and $Y_i=Y_{U_i}$. Then there are $G$-biholomorphisms $\Psi_i\colon X_i\to Y_i$ over the identity of $U_i$ such that $\Psi_i\inv\Psi_j=\Psi_{ij}$. 
 
 Here is an analogue of the twist construction in Galois cohomology. We leave the proof to the reader.
 
 \begin{lemma}\label{lem:reduceId}
 Let $\Psi_{ij}\in Z^1(Q,\A)$  and $\Phi_{ij}\in Z^1(Q,\A_c)$   be cocycles for the open cover $\{U_i\}$ of $Q$.  Let $Y$ and $\Psi_i\colon X_i\to Y_i$ be as above.  The mapping $\Phi_{ij}\mapsto\Psi_i\Phi_{ij}\Psi_j\inv$ induces an isomorphism of $H^1(Q,\A_c)$ and  $H^1(Q,\A_c^Y)$ which sends the class $\Psi_{ij}$ to the trivial class of $H^1(Q,\A_c^Y)$.
 \end{lemma}

 \begin{corollary}\label{cor:onetermholomorphic}
 Let $\Phi_{ij}(t)$ be a homotopy of   cocycles with values in $\A_c(U_i\cap U_j)$ where $\{U_i\}$ is an open cover of $Q$. Suppose that $\Phi_{ij}(0)$ is holomorphic. Then there are   $c_i\in\AA_c(U_i)$ such that  $\Phi_{ij}(t)=c_i(t)\inv\Phi_{ij}(0)c_j(t)$ for all $t$. 
 \end{corollary}
 
 \begin{proof}
 By Lemma \ref{lem:reduceId} we may reduce to the case that $\Phi_{ij}(0)$ is the identity, so we can apply Theorem \ref{thm:mainAc}.
 \end{proof}
 
 Let $X$, $Y$ and the $\Psi_i$ be as above. We say that a $G$-homeomorphism $\Phi\colon X\to Y$ is \emph{strong\/} if $\Psi_i\inv\circ\Phi\colon X_i\to X_i$ is strong for all $i$, i.e., in $\A_c(U_i)$. It is easy to see that this does not depend upon the particular choice of the $\Psi_i$. Similarly one can define what it means for a family $\Phi(t)$ of strong $G$-homeomorphisms to be continuous, $t\in[0,1]$. Then we have the following nice result   \cite[Theorem 1.4]{KLSOka}.
 
 \begin{theorem}\label{thm:mainKLSOka}
 Let $\Phi\colon X\to Y$ be strongly continuous. Then there is a continuous family $\Phi(t)$ of strong $G$-homeomorphisms from $X$ to $Y$ with $\Phi(0)=\Phi$ and $\Phi(1)$ holomorphiic.
 \end{theorem}

\begin{proof}[Proof of Theorem \ref{thm:A}]
We have $\Phi_{ij}$, $\Psi_{ij}\in Z^1(Q,\A)$ and   $c_i\in\A_c(U_i)$ satisfying $\Phi_{ij}=c_i \Psi_{ij}c_j\inv$. Using Lemma \ref{lem:reduceId} we may assume that $\Phi_{ij}$ is the trivial class. Then the $c_i$ are the same thing as a strong $G$-homeomorphism $\Theta\colon X\to Y$ where $Y$ is the Stein $G$-manifold corresponding to the $\Psi_{ij}$(after our twisting). By Theorem \ref{thm:mainKLSOka} not only are there $d_i\in\A(U_i)$ such that $\Psi_{ij}=d_i d_j\inv$, but the $d_i$ are $e_i(1)$ where $e_i(t)$ is a path in $\A_c(U_i)$ starting at $c_i$ and ending at $d_i$. The $d_i$ correspond to a $G$-biholomorphism of $X$ and $Y$ over $Q$. 
\end{proof}

We now prepare to prove Theorem \ref{thm:D}.

\begin{lemma}\label{lem:holomhomotopy}
Let $\Phi\in\AA_c(Q)$ such that $\Phi(1)$ is holomorphic. Then $\Phi$ is homotopic to $\Phi'\in\AA(Q)$ where $\Phi'(1)=\Phi(1)$.
\end{lemma}

\begin{proof}
We have to make use of a sheaf of groups $\F$ on $Q$ which is a subsheaf of the sheaf of  $G$-diffeomorphisms  of $X$ which induce the identity on $Q$ and are algebraic isomorphisms on the fibers of $p$. See \cite[Ch.\ 6]{KLSOka}. We give $\F(U)$   the usual $\ci$-topology.  Let $\FF(U)$ denote the sheaf of homotopies $\Psi(t)$ of elements of $\F(U)$, $t\in[0,1]$, where $\Psi(0)$ is the identity and $\Psi(1)$ is holomorphic. Then \cite[Theorem 10.1]{KLSOka} tells us that $\FF(Q)$ is pathwise connected. Hence for  $\Phi\in\FF(Q)$   there is a homotopy $\Psi(s)\in\FF(Q)$ such that $\Psi(0)=\Phi$ and $\Psi(1)$ is the identity. Then $\Psi(s)$ evaluated at $t=1$ is a homotopy from $\Phi(1)$ to the identity in $\A(Q)$, establishing the lemma when $\Phi\in\FF(Q)$.

 We now use a standard trick. 
Let $\Delta$ denote a disk in $\C$ containing $[0,1]$ with trivial $G$-action. Then $\Delta\times X$ has quotient $\Delta\times Q$ with the obvious quotient mapping. Let  $\rho\colon \Delta\to [0,1]$ be continuous such that $\rho$ sends a neighborhood of $0$ to $0$ and a neighborhood of $1$ to $1$. For $(z,x)\in \Delta\times X$, define $\Psi(z,x)=(z,\Phi(\rho(z),x))$. Then $\Psi\in\tilde\A_c(\Delta\times Q)$ where $\tilde\A_c=\A_c^{\Delta\times X}$. Moreover, $\Psi$ is the identity on the inverse image of a neighborhood of $\{0\}\times Q$ and is holomorphic on  the inverse image of a neighborhood of $\{1\}\times Q$.   By \cite[Theorem 8.7]{KLSOka}   we can find a homotopy $\Psi(s)$ which starts at $\Psi$ and ends up in $\F(\Delta\times Q)$. Moreover, the proof shows that we can assume that the elements of the homotopy are unchanged  over a neighborhood of $\{0,1\}\times Q$. 
Restricting $\Psi(1)$ to $[0,1]\subset \Delta$ we have an element in $\FF(Q)$ which at time 1 is still $\Phi(1)$. Then we can apply the argument above.
\end{proof}

\begin{proof}[Proof of Theorem \ref{thm:D}]
By Lemma \ref{lem:reduceId} we may assume that $\Phi_{ij}(0)$ is the identity cocycle. Since $H^1(Q,\AA_c)$ is trivial, there are $c_i\in  \AA_c(U_i)$ such that $\Phi_{ij}(t)=c_i(t)c_j(t)\inv$ for $t\in[0,1]$. Now the $c_i(1)$ define a strongly continuous $G$-homeomorphism from $X$ to the Stein $G$-manifold $Y$ corresponding to $\Phi_{ij}(1)$. By Theorem \ref{thm:mainKLSOka} there is a homotopy $c_i(t)$, $1\leq t\leq 2$, such that the $c_i(2)$ are holomorphic and split $\Phi_{ij}(1)$. Reparameterizing, we may reduce to the case that the original $c_i(t)$ are holomorphic for $t=1$. Now apply  Lemma \ref{lem:holomhomotopy} to $\Psi_{ij}(t)=c_i(t)c_j(t)\inv$.
\end{proof}

\section{$H^1(Q,\A)\to H^1(Q,\A_c)$ is a bijection}\label{sec:proveB}

We give a proof Theorem  \ref{thm:B}. We are given an open cover $\{U_i\}$ of $Q$ and $\Phi_{ij}\in Z^1(Q,\A_c)$. We want to find $c_i\in\A_c(U_i)$ such that $c_i\inv\Phi_{ij} c_j$ is holomorphic.
We may assume that the $U_i$ are relatively compact, locally finite and Runge.   We say that an open set $U\subset Q$ is \emph{good} if there are sections $c_i\in \A_c(U_i\cap U)$ such that $c_i\inv\Phi_{ij}c_j$ is holomorphic on $U_{ij}\cap U$ for all $i$ and $j$ where $U_{ij}$ denotes $U_i\cap U_j$. This says that $\{\Phi_{ij}\}$ is cohomologous to a holomorphic cocycle on $U$. The goal is to show that $Q$ is good. It is obvious that small open subsets of $Q$ are good.

\begin{lemma}\label{lem:2opensetsgood}
Suppose that $Q=Q'\cup Q''$ where $Q'$ and $Q''$ are good and $Q'\cap Q''$ is Runge in $Q$. Then $Q$ is good.
\end{lemma}

\begin{proof}
 By hypothesis, we have  $ c_i'\in\A_c(Q'\cap U_i)$ and $c_i''\in\A_c(Q''\cap U_i)$ such that 
$$
\Psi'_{ij}=(c_i')\inv \Phi_{ij}c'_j,\text{ and }\Psi''_{ij}=(c''_i)\inv\Phi_{ij}c''_j\text{ are holomorphic}.
$$
 Then on $U_{ij}\cap Q'\cap Q''$ we have
 $$
 \Psi_{ij}''=h_i\inv\Psi'_{ij}h_j\text{ where }h_i=(c_i')\inv c_i''.
 $$
 The $\Psi_{ij}'$ are a holomorphic cocycle for the covering $U_i\cap Q'$ of $Q'$, hence they correspond to a Stein $G$-manifold $X'$ with quotient $Q'$. Similarly the $\Psi_{ij}''$ give us $X''$, and  $X'$ and $X''$ are locally $G$-biholomorphic to $X$ over $Q'$ and $Q''$, respectively. 
The $h_i$   give us a strong $G$-homeomorphism $h\colon X'\to X''$, everything being taken over $Q'\cap Q''$. By Theorem \ref{thm:mainKLSOka} there is a homotopy $h(t,x)$ with $h(0,x)=h(x)$ and $h(1,x)$ holomorphic. Let $k(x)$ denote $h(1,x)$. Then $k$ corresponds to a family $k_i$ homotopic to the family $h_i$.
 
 Now just consider the space $U_i$ covered by the two open sets $U_i\cap Q'$ and $U_i\cap Q''$. Then $h_i$ and $k_i$ are defined on the intersection of the two open sets and are homotopic where $h_i$ is cohomologous to the trivial cocycle since $h_i=(c_i')\inv c_i''$. By Corollary \ref{cor:onetermholomorphic} and Theorem \ref{thm:A} the cohomology class represented by $k_i(x)$ is holomorphically trivial. Hence there are holomorphic sections $h_i'$ and $h_i''$ such that $k_i=(h_i')\inv h_i''$ on $U_i\cap Q'\cap Q''$. Then $h_i'\Psi_{ij}'(h_j')\inv=h_i''\Psi_{ij}''(h''_j)\inv$ on $U_{ij}\cap Q'\cap Q''$. We construct a holomorphic cocycle $\Psi_{ij}$ on $U_{ij}$ by $\Psi_{ij}=h'_i\Psi_{ij}'(h_j')\inv$ on $U_{ij}\cap Q'$ and $h''_i\Psi_{ij}''(h''_j)\inv$ on $U_{ij}\cap Q''$.
 
Using Lemma \ref{lem:reduceId} we may reduce to the case that $\Psi_{ij}$ is the trivial cocycle. As in the beginning of the proof there are $c_i'\in\A_c(Q'\cap U_i)$ and $c_i''\in\A_c(Q''\cap U_i)$ such that 
$$
\Phi_{ij}|_{X'}=c_i' (c'_j)\inv\text{ and }\Phi_{ij}|_{X''}=c''_i(c''_j)\inv
$$
where $X'=X_{Q'}$ and $X''=X_{Q''}$.
Let $h_i=(c_i')\inv c_i''$. Then $h_i=h_j$ on $U_{ij}\cap Q'\cap Q''$, hence we have a section  $h\in\A_c(Q'\cap Q'')$, and this section gives the same cohomology class as $\Phi_{ij}$ (use the open cover $\{Q'\cap U_i,Q''\cap U_i\}$). By Theorem \ref{thm:mainKLSOka}, $h$ is homotopic to  an element $\tilde h\in\A(Q'\cap Q'')$, and this holomorphic section gives the same cohomology class by Corollary \ref{cor:onetermholomorphic}.   Since going to a refinement of an open cover is injective on $H^1$, we see that our original $\Phi_{ij}$ differs from a holomorphic cocycle by a coboundary. Thus $Q$ is good.
\end{proof}
 
 \begin{proof}[Proof of Theorem \ref{thm:B}]
 Using Lemma \ref{lem:2opensetsgood} as in \cite[\S 5]{Cartan58} we can show that there is a cover of $Q$ by compact subsets $K_n$ such that $K_1\subset V_2\subset K_2\dots$ where $V_j$ is the interior of $K_j$ and such that  a neighborhood of every $K_n$ is good.    We can assume that $U_i\cap V_n\neq \emptyset$ implies that $U_i\subset V_{n+1}$. This is possible by replacing $\{K_n\}$ by a subsequence. For each $n$ we choose  $c_i^n\in\A_c(U_i\cap V_n)$ such that 
   $$
   (c_i^n)\inv\Phi_{ij}c_j^n=\Psi^n_{ij}\text{ is holomorphic on }U_{ij}\cap V_n.
   $$
Then
$$
\Psi_{ij}^n=(d_i^n)\inv\Psi_{ij}^{n+1}d_j^n\text{ on }U_{ij}\cap  V_n
$$  
 where $d_i^n=(c_i^{n+1})\inv c_i^n$ gives a strongly continuous map   from  the Stein $G$-manifold $Y_n$ over $V_n$ obtained using the $\Psi_{ij}^n$ to the Stein $G$-manifold $Y_{n+1}$ obtained using the $\Psi_{ij}^{n+1}$. We know that   the map is homotopic to a holomorphic one. Hence there are homotopies $d_i^n(t)$ on $U_i\cap V_n$ such that
 \begin{enumerate}
\item $\Psi_{ij}^n=(d_i^n(t))\inv \Psi_{ij}^{n+1} d_j^n(t)$ on $U_i\cap U_j\cap V_n$, for all $t$.
\item $d_i^n(0)=(c_i^{n+1})\inv c_i^n$.
\item The $d_i^n(1)$ give  a $G$-equivariant biholomorphic map from $Y_n$ to $Y_{n+1}$ over $\Id_{V_n}$.
 \end{enumerate}
Without changing the $c_i^n$ we may replace the $c_i^{n+1}$ by sections $\tilde c_i^{n+1}$ such that
\begin{enumerate}
\addtocounter{enumi}{3}
\item $\widetilde \Psi_{ij}^{n+1}=(\tilde c_i^{n+1})\inv\Phi_{ij}\tilde c_j^{n+1}$ is holomorphic on $U_i\cap U_j\cap V_{n+1}$.
\item $\tilde c_i^{n+1}=c_i^n$ on $U_i\cap V_{n-2}$.
\end{enumerate}
It suffices to set $\tilde c_i^{n+1}=c_i^{n+1}$ if $U_i\cap V_{n-1}=\emptyset$ and if not, then $U_i\subset V_n$, and we can set
$$
\tilde c_i^{n+1}=c_i^{n+1}\cdot d_i^n(\lambda(x)),
$$
where $\lambda\colon V_n\to [0,1]$ is continuous, $0$ for $x\in V_{n-2}$ and 1 for $x\not\in V_{n-1}$. Then one has (4) and (5). Thus we can arrange  that $c_i^{n+1}=c_i^n$ in $U_i\cap V_{n-2}$, hence we obviously have convergence of the $c_i^n$ to a continuous section $c_i$ such that $(c_i)\inv\Phi_{ij}c_j$ is holomorphic. 
 \end{proof}
 
 We now have Theorems \ref{thm:A} and \ref{thm:B} which imply Corollary \ref{cor:main}, i.e., that  $H^1(Q,\A)\to H^1(Q,\A_c)$ is a bijection.
 
 \begin{proof}[Proof of Theorem \ref{thm:C}]
 This is immediate from Corollary \ref{cor:onetermholomorphic} and Theorem \ref{thm:B}
 \end{proof}
 
 We end with the analogue of an approximation theorem of Grauert.
 \begin{theorem}
 Let $U\subset Q$ be Runge. Suppose that $\Phi\colon X_U\to Y_U$ is biholomorphic and $G$-equivariant inducing $\Id_U$. Here $X$ and $Y$ are locally $G$-biholomorphic over $Q$. Then $\Phi$ can be arbitrarily closely approximated by $G$-biholomorphisms of $X$ and $Y$ over $\Id_Q$ if and only if this is true for strong $G$-homeomorphisms of $X$ and $Y$.
 \end{theorem}

\begin{proof}
Let $K\subset U$ be compact and let $\Phi\in\Mor(X_U,Y_U)^G$ be our holomorphic $G$-equivariant map inducing $\Id_U$. We can find a relatively compact open  subset $U'$ of $U$ which contains $K$ and is Runge in $Q$. By hypothesis, there is a strong $G$-homeomorphism $\Psi\colon X\to Y$ which is arbitrarily close to $\Phi$ over $\overline{U'}$. Then $\Psi\inv\Phi=\exp D$ where $D\in\LAc(U')$, hence $\Psi'$ and $\Phi'$ are homotopic, where $\Phi'$ is the restriction of $\Phi$ to $U'$ and similarly for $\Psi'$.   Now $\Psi$ is homotopic to a biholomorphic $G$-equivariant map $\Theta\colon X\to Y$ inducing $\Id_Q$, and  $\Psi'$ is homotopic to the restriction $\Theta'$ of $\Theta$ to $U'$. Then $(\Phi')\inv\Theta'$ is holomorphic and homotopic to the identity section over $U'$. Since the end points of the homotopy are holomorphic, by Theorem \ref{thm:D} we can find a homotopy all of whose elements are holomorphic. 
 By \cite[Theorem 10.1]{KLSOka} there is a section $\Delta\in\A(Q)$ which is arbitrarily close to $(\Phi')\inv\Theta'$ on $U'$.  Then $\Theta\Delta\inv$, restricted to $U'$, is arbitrarily close to $\Phi'$, hence this is true over $K$. This establishes the theorem.
 \end{proof}

  \bibliographystyle{amsalpha}
\bibliography{Oka.paperbib}

  \end{document}